\documentclass[]{article}
\usepackage{amsfonts}
\usepackage{amssymb}
\usepackage[letterpaper, left=2.5cm, right=2.5cm, top=2.5cm,
bottom=2.5cm,dvips]{geometry}
\usepackage{verbatim}

\usepackage[T1]{fontenc}
\usepackage{graphicx}
\usepackage{diagbox}
\usepackage{xcolor}
\usepackage{amsmath}
\usepackage{array,multirow}
\usepackage{dsfont}
\usepackage{amsthm}
\usepackage{enumerate}

\usepackage[pagebackref]{hyperref}
\hypersetup{
	colorlinks,
	linkcolor={blue!60!black},
	citecolor={green!60!black},
	urlcolor={green!60!black}
}

\newtheorem{theorem}{Theorem}

\newtheorem{claim}{Claim}

\newtheorem{conjecture}{Conjecture}
\newtheorem{corollary}{Corollary}

\newtheorem{definition}{Definition}

\newtheorem{lemma}{Lemma}

\newtheorem{proposition}{Proposition}

\numberwithin{equation}{section}

\DeclareMathOperator{\wdesc}{wdesc}
\DeclareMathOperator{\desc}{desc}
\DeclareMathOperator{\mx}{mx}

\usepackage{xcolor}

\begin{document}
\title{Improved upper bounds for wide-sense frameproof codes}
\date{}
\author{Yuhao~Zhao\thanks{Y. Zhao ({\tt zhaoyh21@mail.ustc.edu.cn}) is with the School of Mathematical Sciences, University of Science and Technology of China, Hefei, 230026, Anhui, China.} ~and~Xiande~Zhang
\thanks{X. Zhang ({\tt drzhangx@ustc.edu.cn}) is with the School of Mathematical Sciences,
University of Science and Technology of China, Hefei, 230026, and with Hefei National Laboratory, University of Science and Technology of China, Hefei 230088, China. 
}
}



\maketitle

\begin{abstract}
Frameproof codes have been extensively studied for many years due to their application in  copyright protection and their connection to extremal set theory.
In this paper, we investigate upper bounds on the cardinality of  wide-sense $t$-frameproof codes.
	 For $t=2$, we apply results from Sperner theory to give a better upper bound, which significantly improves a recent bound by Zhou and Zhou. For  $t\geq 3$, we provide a general upper bound by establishing a relation between wide-sense frameproof codes and cover-free families. Finally, when the code length $n$ is at  most $\frac{15+\sqrt{33}}{24}(t-1)^2$, we show that a wide-sense $t$-frameproof code has at most $n$ codewords, and the unique optimal code consists of all weight-one codewords. As byproducts, our results improve several best known results on binary $t$-frameproof codes.


\end{abstract}

\section{Introduction}
   Frameproof codes were first introduced by Boneh and Shaw~\cite{BS} in the context of digital fingerprinting. Let $Q:=\left\lbrace 0,1,\dots,q-1\right\rbrace$ be an alphabet of size $q$. The \textit{fingerprints}, which are generally viewed as codewords in $Q^n$, are distributed to all registered users to protect copyrighted digital products. The clients do not know the locations and symbols embedded in the data, so they cannot remove or modify them. However, a coalition of some clients could share and compare their copies, so that they could recover the locations and symbols of the fingerprints to produce an illegal copy. Frameproof codes are designed to prevent a small coalition of clients from constructing a copy of fingerprint of an innocent user (a user not in the coalition). There is a lot of work on frameproof codes and their applications, see for instance \cite{Barg,Blackburn2003,BS,chee,Cheng-Miao,Guo-Stinson-Trung,Shangguan2017,Staddon,Stinson-Wei1998,Trung}, and references therein. One of the central problems on this topic is the studying of upper and lower bounds on the largest cardinality of frameproof codes.

   Note that there are different variants of frameproof codes in the literature, depending on different definitions of descendant sets. In this paper, we concentrate on frameproof codes in the wide-sense model which were first adopted by Boneh and Shaw~\cite{BS}. We assume the \textit{Marking Assumption} which restricts the capability of a coalition: the members of the coalition can only alter those coordinates of the fingerprint in which at least two of their fingerprints differ, as stated in \cite{Barg}.
   We denote $\mathbf{c}:=(c_1,\dots,c_n)$ as a codeword in $Q^n$.
    Given a $t$-set $X=\left\lbrace \mathbf{c}^1,\dots,\mathbf{c}^t\right\rbrace \subseteq Q^n$,
   we say a bit position $i\in[n]$ is {\it undetectable} for $X$ if $c_i^1=\dots=c_i^t$.
   Let $U(X)$ be the set of undetectable bit positions for $X$. Note that when $t=1$, any $i\in[n]$ is undetectable.

   The {\it wide-sense descendant set} of a $t$-set $X=\left\lbrace \mathbf{c}^1,\dots,\mathbf{c}^t\right\rbrace \subseteq Q^n$ is defined by
   \begin{equation*}
   	 \wdesc(X) = \left\lbrace \mathbf{y}\in Q^n : y_i=c_i^1 \ \text{if}\ i\in U(X)\right\rbrace,
   \end{equation*}
   in contrast to the {\it narrow-sense descendant set}
   \begin{equation*}
   \desc(X) = \left\lbrace \mathbf{y}\in Q^n : y_i \in \left\lbrace c_i^1,\dots,c_i^t\right\rbrace \right\rbrace.
   \end{equation*}
   The wide-sense descendant set $\wdesc(X)$ represents the set of fingerprints that the coalition can construct from $X$ according to the marking assumption, which has been studied under the names of envelope \cite{Barg},
   feasible set~\cite{BS,Stinson-Wei1998}, or  Boneh-Shaw descendant~\cite{Blackburn-survey}.

   \begin{definition}\label{orgde} Let $t\geq 2$ be an integer.
   	 We define a code $\mathcal{C}\subseteq Q^n$ to be a \emph{wide-sense $t$-frameproof code} if
   	 $
   	   \wdesc(X) \cap \mathcal{C} = X
   	 $
   	 for all $X\subseteq \mathcal{C}$ with $|X|\leq t$. And we define a code $\mathcal{C}\subseteq Q^n$ to be a \emph{narrow-sense $t$-frameproof code}, or \emph{a $t$-frameproof code}, if
   	 $
   	 \desc(X) \cap \mathcal{C} = X
   	 $
   	  for all $X\subseteq \mathcal{C}$ with $|X|\leq t$.
   \end{definition}

	Note that $X \subseteq \mathrm{desc}(X) \subseteq \mathrm{wdesc}(X)$, so a wide-sense $t$-frameproof code is also a $t$-frameproof code. When $q=2$, $\mathrm{desc}(X) = \mathrm{wdesc}(X)$, so a binary $t$-frameproof code is also a wide-sense $t$-frameproof code. However, $\mathrm{wdesc}(X)$ always strictly contains $\mathrm{desc}(X)$ when $2\leq |X|< q$,
	which is one reason why the problem of constructing such wide-sense codes is more difficult than the original problem, as mentioned by Blackburn~\cite{Blackburn-survey}. In this paper, we mainly investigate the upper bounds for wide-sense frameproof codes.


\subsection{Wide-sense 2-frameproof codes}
    For wide-sense $2$-frameproof codes, Stinson and Wei~\cite{Stinson-Wei1998} used Sperner's theorem to show that the size of a wide-sense $2$-frameproof code of length $n$ is at most {$\binom{n}{\lceil n/2 \rceil}+1$} for any alphabet $Q$. Later, Panoui~\cite{panoui} improved upon their result to show that the size of a wide-sense $2$-frameproof code 
    is at most $\binom{n}{\frac{n}{2}-1}+1$ for even length $n$ and at most $\binom{n}{\frac{n-1}{2}}-\frac{n-1}{2}$ for odd length $n$. Recently, Zhou and Zhou~\cite{Zhou-Zhou} obtained the following better bounds.
    \begin{theorem}[\cite{Zhou-Zhou}]\label{thm:zhou-zhou}
    	Let $\mathcal{C}\subseteq Q^n$ be a wide-sense $2$-frameproof code of size $m$.
    	\begin{itemize}
    		\item [i)] If $n\geq 8$ is even, then $m\leq \binom{n}{\frac{n}{2}-1}-\frac{n}{2}+1$.
    		\item [ii)] If $n\geq 7$ is odd, then
    		\begin{equation*}
    		  m \leq \left\{
    		    \begin{aligned}
    		      &\binom{n}{\frac{n-1}{2}}- \frac{n^2-9}{8}-\left\lfloor\frac{(n-5)^2}{64}\right\rfloor,
    		            &\text{ if}\ n\equiv 1\ (\mathrm{mod}\ 4) ,\\
    		      &\binom{n}{\frac{n-1}{2}}-\frac{(n+1)^2-8}{8}-\left\lfloor\frac{(n-3)^2}{64}\right\rfloor,  &\text{ if}\ n\equiv 3\ (\mathrm{mod}\ 4).
    		    \end{aligned}
    		  \right.
    		\end{equation*}
    	\end{itemize}
    \end{theorem}

    Let $m_q(n,t)$ be the maximum size of a $q$-ary wide-sense $t$-frameproof code in $Q^n$, and let $g(q,n):=\binom{n}{\lfloor\frac{n-1}{2}\rfloor}-m_q(n,2)$.  Then the upper bounds in Theorem~\ref{thm:zhou-zhou} imply polynomially lower bounds for $g(q,n)$. 
   	Our first result (Theorem~\ref{thm:2-wFP-main}) provides a better upper bound for wide-sense $2$-frameproof codes, which implies an exponentially lower bound for $g(q,n)$ when $n$ is sufficiently large (see Corollary~\ref{thm:2-wFP-main-cor}).

      \begin{theorem}\label{thm:2-wFP-main}
      For any $\epsilon\in(0,\frac{1}{2})$, we have
      	\begin{equation*}
      	m_q(n,2) \leq
      	\max \left\lbrace
      	\sum_{3\lceil \frac{1-\epsilon}{2}n \rceil -n-3 \leq   i   \leq \lceil \frac{1-\epsilon}{2}n\rceil -1}
      	\binom{n}{i} +1,
      	\left(1-\frac{1}{\binom{n}{\lfloor \epsilon n\rfloor}}\right)\binom{n}{\lfloor\frac{n-1}{2}\rfloor}+2 \right\rbrace
      	\triangleq \phi(n,\epsilon).
      	\end{equation*}
      \end{theorem}

      		Stirling's formula shows that $\binom{n}{\lfloor\frac{n-1}{2}\rfloor} \geq \delta \cdot \frac{2^n}{\sqrt{n}}$ for some constant $\delta\in (0,1)$.      	
      	Let $H(x)$ be the binary entropy function and write $H^{-1}(x)$ for its inverse restricted to $[0,1/2]$.
      	For $n\geq 2(1-\log_2 \delta )/(1-H(1/4))+30 $, we define
      	\[
      	\epsilon(n): = \max \left\lbrace \frac{1}{n}, 1- 2 H^{-1}\left( 1-\frac{2+\log_2 n - 2\log_2 \delta}{2n} \right)
      	\right\rbrace \in \left( 0,1/2\right) ,
      	\]
      	which tends to 0 as $n\rightarrow \infty$. Observe that
      	\begin{equation*}\label{equa-1}
      	\sum_{3\lceil \frac{1-\epsilon(n)}{2}n \rceil -n-3 \leq   i   \leq \lceil \frac{1-\epsilon(n)}{2}n\rceil -1} \binom{n}{i} +1
      	\leq
      	2^{nH\left( \frac{1-\epsilon(n)}{2}\right) }+1
      	\leq \frac{1}{2} \binom{n}{\lfloor\frac{n-1}{2}\rfloor}+2
      	\leq
      	\left(1-\frac{1}{\binom{n}{\lfloor \epsilon(n) n\rfloor}}\right)\binom{n}{\lfloor\frac{n-1}{2}\rfloor}+2.
      	\end{equation*}
      	It follows that
      	$$
      	\phi(n,\epsilon(n)) = \left(1-\frac{1}{\binom{n}{\lfloor \epsilon(n) n\rfloor}}\right)\binom{n}{\lfloor\frac{n-1}{2}\rfloor}+2
      	\leq \binom{n}{\lfloor\frac{n-1}{2}\rfloor}+2 - \delta \cdot \frac{2^{n(1-H(\epsilon(n)))}}{\sqrt{n}},
      	$$ where we used $\binom{n}{\lfloor \epsilon(n) n\rfloor} \leq 2^{nH(\epsilon(n))}$.
      	Thus we have the following corollary which shows that our upper bound for wide-sense $2$-frameproof codes improves Theorem~\ref{thm:zhou-zhou} substantially for large $n$.
      	\begin{corollary}\label{thm:2-wFP-main-cor}
      For $n\geq 2(1-\log_2 \delta )/(1-H(1/4))+30 $,
      		\begin{equation*}
      		m_q(n,2) \leq \phi(n,\epsilon(n))
      		\leq \binom{n}{\lfloor\frac{n-1}{2}\rfloor}+2 - \delta \cdot \frac{2^{n(1-H(\epsilon(n)))}}{\sqrt{n}}.
      		\end{equation*}Since $\lim_{n \rightarrow \infty} H(\epsilon(n)) = 0$, we see that $g(q,n)$  is exponentially large for large $n$.
      	\end{corollary}

\subsection{General bounds for wide-sense $t$-frameproof codes}
   By establishing a relationship between binary $t$-frameproof codes and Sperner families, Stinson and Wei~\cite{Stinson-Wei1998} proved that if $\mathcal{C}\subseteq \left\lbrace 0,1\right\rbrace^n$ is a $t$-frameproof code of size $m$, then
     $m \leq t-1+ \binom{n-t+2}{\lceil (n-t+2)/2 \rceil}$.
    Since they indeed considered the wide-sense model in their proof, it is straightforward to get the following extension for wide-sense $t$-frameproof codes with arbitrary alphabet size.
    \begin{theorem}[\cite{Stinson-Wei1998}]\label{thm:Stinson-wei1998}
    For any $q\geq 2$,
    	\begin{equation*}
m_q(n,t) \leq t-1+ \binom{n-t+2}{\lceil \frac{n-t+2}{2} \rceil}.
    	\end{equation*}
    \end{theorem}

    Note that the upper bound in Theorem~\ref{thm:Stinson-wei1998} is independent of $q$. Since wide-sense $t$-frameproof codes are also $t$-frameproof codes, the upper bounds of $t$-frameproof codes are  upper bounds of wide-sense $t$-frameproof codes.
    The following general upper bound is due to Blackburn~\cite{Blackburn2003}.
    \begin{theorem}[\cite{Blackburn2003}]\label{thm:Blackburn2003}
    If $\mathcal{C}\subseteq Q^n$ is a $t$-frameproof code of maximum size $m$, then\begin{equation*}
    	m_q(n,t)\leq m \leq \left( \frac{n}{n-(r-1)\lceil n/t\rceil} \right) q^{\lceil n/t\rceil} +O(q^{\lceil n/t\rceil-1}),
    	\end{equation*} where $r\in [t]$ is an integer satisfying $r\equiv n\pmod t$.
    \end{theorem}
    When $q< 2^t$, the bound in Theorem~\ref{thm:Blackburn2003} viewed as a function of $n$ is better than that in Theorem~\ref{thm:Stinson-wei1998}.
    Shangguan et al.~\cite{Shangguan2017} further improved this bound when $cq\leq t$ for some constant $c$.
    \begin{theorem}[\cite{Shangguan2017}]\label{thm:shangguan2017-upper-bound}
    	Suppose $\mathcal{C}\subseteq Q^n$ is a $t$-frameproof code of maximum size $m$. Then we have
    	\begin{equation*}
    	m_q(n,t)\leq m \leq \binom{n}{\lceil \frac{n(q-1)}{\binom{t}{2}}\rceil} q^{\lceil \frac{n(q-1)}{\binom{t}{2}}\rceil} +t
    	 \leq  q^{\lceil \frac{n(q-1)}{\binom{t}{2}}\rceil \log_q \frac{eq\binom{t}{2}}{q-1}} +t.
    	\end{equation*}
    \end{theorem}

   When $q=2$, the bound in Theorem~\ref{thm:shangguan2017-upper-bound} is superior to that in Theorem~\ref{thm:Blackburn2003} for $t\geq 25$.

Our second result provides a better general upper bound for wide-sense $t$-frameproof codes with $t\geq 3$.

      \begin{theorem}\label{thm:general-upper-bound}
      	 Let $t\geq 3$. Then
      	 \begin{equation*}
      	  m_q(n,t) \leq \binom{n}{\lceil \frac{n-t+1}{\binom{t}{2}} \rceil} +t.
      	 \end{equation*}
      \end{theorem}
  {\color{black}
  In particular, since a binary $t$-frameproof code  is also a wide-sense $t$-frameproof code, our bound in Theorem~\ref{thm:general-upper-bound} exponentially improves Theorem~\ref{thm:shangguan2017-upper-bound} for $q=2$ and any fixed $t\geq 3$. In fact, when $q=2$ and $n$ is large,  the upper bound in Theorem~\ref{thm:shangguan2017-upper-bound} is $O\left( \binom{n}{\lceil n/\binom{t}{2}\rceil} 2^{\lceil n/\binom{t}{2}\rceil}\right) $ while our bound is $O\left( \binom{n}{\lceil (n-t+1)/\binom{t}{2}\rceil}\right)$. We state it below.
}
     \begin{corollary}
     	 Let $t\geq 3$. Suppose $\mathcal{C}\subseteq \left\lbrace 0,1\right\rbrace^n$ is a $t$-frameproof code of size $m$. Then
     	 $
     	 m \leq \binom{n}{\lceil (n-t+1)/\binom{t}{2} \rceil} +t.
     	 $
     \end{corollary}

   Moreover, observe that
   \begin{equation*}
    \binom{n}{\lceil \frac{n-t+1}{\binom{t}{2}} \rceil} +t
    \leq
    \left( e\binom{t}{2}\right)^{\lceil\frac{n}{\binom{t}{2}}\rceil} +t
    =
    2^{\lceil \frac{n}{\binom{t}{2}}\rceil \log_2 \left( e\binom{t}{2}\right) }  +t
    \leq
    2^{ \left( \frac{4 \log_2 t +O(1)}{t^2} \right) n} +t.
   \end{equation*}
   Thus $\limsup_{n\rightarrow \infty} {\log_2 m_q(n,t)}/{n} \leq   (4+o(1))\frac{\log_2 t}{t^2}.$
   Actually, by using our proof and the better bound on cover-free families given by D'yachkov and Rykov~\cite{D'yachkov}, this asymptotic bound can be slightly improved as follows.

    \begin{theorem}\label{thm:general-upper-bound-2}
     For any given $q\geq 2$, we have
    \begin{equation*}
   \limsup_{n\rightarrow \infty} \frac{\log_2 m_q(n,t)}{n} \leq   \left( 2+o(1)\right) \frac{\log_2 t}{t^2}.
    \end{equation*}
    In particular, in the binary case this gives an upper bound for binary $t$-frameproof codes.
     \end{theorem}

\subsection{Tight bounds for wide-sense $t$-frameproof codes}
      Let $\mathcal{C}\subseteq Q^n$ be a wide-sense $t$-frameproof code, and consider the representation matrix of $\mathcal{C}$ which is an $n\times |\mathcal{C}|$ matrix on $q$ symbols with columns corresponding to codewords in $\mathcal{C}$. For any given wide-sense $t$-frameproof code, we may derive new ones from it by simply permuting the elements $0, 1,\dots, q-1$ in each row separately. Such codes can be considered to be equivalent, and hence we could focus on a standard
      representative. When $q=2$, we say that the representation matrix of $\mathcal{C}\subseteq \left\lbrace 0,1\right\rbrace^n$ is in \emph{standard form} if every row has at most $|\mathcal{C}|/2$ entries of $1$.

     Note that  any permutation matrix represents a binary (wide-sense) $t$-frameproof code for any $t\geq 2$. Guo et al.~\cite{Guo-Stinson-Trung} showed that this is indeed an optimal code  for certain $n$ and proved the uniqueness of  optimal codes.
      \begin{theorem}[\cite{Guo-Stinson-Trung}]\label{thm:Guo-Stinson-Trung}
      	Let $t, n$ be positive integers such that $t \geq 3$
      	and $3 \leq n \leq 3t$. Suppose there exists a binary $t$-frameproof code $\mathcal{C}\subseteq \left\lbrace 0,1\right\rbrace^n$ of size $m$. Then  $m \leq n.$ Equality holds if and only if the representation matrix of $\mathcal{C}$ in standard form is a permutation matrix of degree $n$.
      	

      \end{theorem}

       In \cite{Shangguan2017}, Shangguan et al. 
       showed that the above results are still true for  $3\leq n< \binom{t+1}{2}$.
       More recently, by relating binary frameproof codes to a conjecture of Erd\H{o}s, Frankl and F\"uredi~\cite{erdos-frankl-furedi} on cover-free families, Ge et al.~\cite{Ge-Shangguan-Wang2019} proved the following result.
       \begin{theorem}[\cite{Ge-Shangguan-Wang2019}]\label{thm:Ge-Shangguan-Wang2019}
       	 Let $t, n$ be positive integers such that $t \geq 3$
       	 and $2 \leq n < \frac{15+\sqrt{33}}{24}(t-2)^2$. Suppose there exists a binary $t$-frameproof code $\mathcal{C}\subseteq \left\lbrace 0,1\right\rbrace^n$ of size $m$. Then  $m \leq n.$      	
       \end{theorem}

       However, Theorem~\ref{thm:Ge-Shangguan-Wang2019} doesn't characterize the optimal codes. Our third result improves and generalizes Theorem~\ref{thm:Ge-Shangguan-Wang2019} to wide-sense frameproof codes, which is stated as follows.

       \begin{theorem}\label{thm:tight-bound}
       	Let $t, n$ be positive integers such that $t \geq 3$
       	and $2 \leq n < \frac{15+\sqrt{33}}{24}(t-1)^2$. Then  $m_q(n,t) \leq n.$      	
       \end{theorem}

      To represent a $q$-ary code $\mathcal{C}\subseteq Q^n$ in a matrix, we can similarly define its standard form.
      For each row $r$ and each $i\in Q$, let $\lambda_i(r)$ be the number of $i$'s in the row $r$. Hence $|\mathcal{C}|=\sum_{i=0}^{q-1}\lambda_i(r)$ for each row $r$. We say that the representation matrix of $\mathcal{C}$ is in standard form if for each row $r$ we have $\lambda_0(r)\geq \lambda_1(r) \geq \dots \geq \lambda_{q-1}(r)$.



       Our last result determines the optimal codes in Theorem~\ref{thm:Ge-Shangguan-Wang2019} and Theorem~\ref{thm:tight-bound}. 

       \begin{theorem}\label{thm:tight-bound-iff}
       	 Let $t, n$ be positive integers such that $t \geq 3$ and $3 \leq n < \frac{15+\sqrt{33}}{24}(t-1)^2$. A code $\mathcal{C}\subseteq Q^n$ is a wide-sense $t$-frameproof code of size $n$ if and only if its representation matrix in standard form is a permutation matrix of degree $n$.
       \end{theorem}

       We remark that the condition $t\geq 3$ in Theorem~\ref{thm:tight-bound} and Theorem~\ref{thm:tight-bound-iff} is necessary. Take the $n\times n$ identity matrix $I_n$ and append a column of $1$'s to it. Clearly this gives a wide-sense $2$-frameproof code, and hence Theorem~\ref{thm:tight-bound} does not hold for $t=2$. Moreover, let
       \begin{equation*}
       A:= \left( \begin{array}{cccc}
       	1 & 1 & 0 & 0 \\
       	0 & 1 & 1 & 0 \\
       	1 & 0 & 1 & 0 \\
       	0 & 0 & 0 & 1 \\
       \end{array} \right)
              \end{equation*}
              and note that for $t=2$ and $n\geq 4$, the $n\times n$ matrix
               $\left( \begin{array}{cc}
             	A & 0 \\
             	0 & I_{n-4} \\
             \end{array} \right)
                          $
        represents a wide-sense $2$-frameproof code.  Thus Theorem~\ref{thm:tight-bound-iff} does not hold for $t=2$ and $n\geq 4$. And it is easy to check that Theorem~\ref{thm:tight-bound-iff} does hold for the last case $t=2$ and $n=3$.
       Moreover, the condition $n\geq 3$ in Theorem~\ref{thm:tight-bound-iff} is also necessary, as one can check directly that when $n= 2$ the representation matrix in standard form can be $\left( \begin{array}{cc}
       1 & 0 \\
       0 & 0 \\
       \end{array} \right) $ which is not a permutation matrix.

 \subsection{Structure of the paper}
      The rest of the paper is organized as follows. In Section~\ref*{section:preliminaries}, we introduce some useful notations and results from extremal set theory. By using several results from Sperner theory, we prove Theorem~\ref{thm:2-wFP-main} in Section~\ref{section:pf-of-2wFP}. In Section~\ref{section:cover-free-and-wFP}, we develop a connection between wide-sense frameproof codes and cover-free families and prove Theorem~\ref{thm:general-upper-bound}.  In Section~\ref{section:pf-of-tight-bound}, we determine optimal cover-free families of size $n$ for certain $n$ and use it to prove Theorem~\ref{thm:tight-bound}. Then we prove a stability result for cover-free families of size $n-1$ in Section~\ref{section:pf-of-tight-bound-iff} and use it to prove Theorem~\ref{thm:tight-bound-iff}. Finally, Section~\ref{section:conclusion} contains some useful discussions and interesting problems.

\section{Preliminaries}\label{section:preliminaries}
{\color{black}Let $[n]$ denote the set $\{1,2,\ldots,n\}$, and let $2^{[n]}$ denote the power set of $[n]$. For any positive integer $ a \leq n$, let $\binom{[n]}{a}$ denote the collection of all $a$-subsets of $[n]$, and let $\binom{[n]}{\leq a}$ denote the collection of all subsets of $[n]$ of size at most $a$. For any set $A\subseteq [n]$, we write $\overline{A}:=[n]\setminus A$, and for a family $\mathcal{F}\subseteq 2^{[n]}$, we write $\overline{{\mathcal{F}}}:=\{\overline{A}: A\in \mathcal{F}\}$.}
      \subsection{Sperner families and shadows}
      Let $\mathcal{F}$ be a family of finite sets. If any two distinct sets in $\mathcal{F}$ are  incomparable, that is,  $A\not\subseteq B$ for any different members $A, B\in\mathcal{F}$, then $\mathcal{F}$ is called an {\it antichain} or  a {\it Sperner family}.
      \begin{theorem}[Sperner's theorem \cite{Sperner}]
      	If $\mathcal{F}\subseteq 2^{[n]}$ is a Sperner family, then \begin{equation*}
      	|\mathcal{F}| \leq \binom{n}{\lfloor n/2\rfloor}.
      	\end{equation*}
      \end{theorem}

      The LYM inequality of Lubell, Meshalkin and Yamamoto~\cite{Lubell,Mesalkin,Yamamoto} is an important result in extremal set theory that was used to prove Sperner's theorem.
      \begin{theorem}[LYM inequality]
      	Let $\mathcal{F}\subseteq 2^{[n]}$ be a Sperner family. Then
      	\begin{equation*}
      	\sum_{A\in \mathcal{F}} \frac{1}{\binom{n}{|A|}} \leq 1.
      	\end{equation*}
      \end{theorem}

      One approach to proving the LYM inequality is based on local LYM inequalities, a version of which dates back to Sperner~\cite{Sperner}.
      If $r\geq 1$ is an integer and $\mathcal{F}\subseteq 2^{[n]}$, we write $\partial^r \mathcal{F}$ for the \emph{$r$-fold shadow} of $\mathcal{F}$, that is, the collection of sets which can be obtained by deleting $r$ elements  from some set in $\mathcal{F}$. The $1$-fold shadow of $\mathcal{F}$ is simply written as $\partial \mathcal{F}$. Similarly, we write $\partial^{-r}\mathcal{F}$ for the collection of sets which can be obtained by adding $r$ elements in $[n]$ to some set in $\mathcal{F}$.
      \begin{lemma}[Local LYM inequality]\label{lemma:local-lym}
      	Let $\mathcal{F}\subseteq \binom{[n]}{a}$ be a family of sets of size $a$. For any $0 \leq r \leq n-a$,
      	\begin{equation*}
      	\frac{|\partial^{-r} \mathcal{F}|}{\binom{n}{a+r}} \geq \frac{|\mathcal{F}|}{\binom{n}{a}}.
      	\end{equation*}
      	For any $0 \leq r \leq a$,
      	\begin{equation*}
      	\frac{|\partial^{r} \mathcal{F}|}{\binom{n}{a-r}} \geq \frac{|\mathcal{F}|}{\binom{n}{a}}.
      	\end{equation*}
      	For both directions, equality holds if and only if $\mathcal{F}=\emptyset$ or $\mathcal{F}=\binom{[n]}{a}$.
      \end{lemma}
      We say a family $\mathcal{F}\subseteq 2^{[n]}$ is {\it $r$-wise intersecting} if $\cap_{A\in \mathcal{F}'}A \neq \emptyset$ for any $\mathcal{F}'\subseteq \mathcal{F}$ with $|\mathcal{F}'|\leq r$. A $2$-wise intersecting family is known as an {\it intersecting family}.
       We need the following result about shadows of intersecting families given by Katona~\cite{Katona64}.
       \begin{lemma}[\cite{Katona64}]\label{lemma-katona}
       	 If $\mathcal{F}\subseteq \binom{[n]}{a}$ is an intersecting family, then $|\partial \mathcal{F}|\ge|\mathcal{F}|$.
       \end{lemma}

  \subsection{Restricted/symmetric differences}
       Let $L$ be a set of positive integers. We denote by $f_L(n)$ the maximum
       size of a family $\mathcal{F} \subseteq 2 ^{[n]}$ such that $|A\Delta B| \in L$ for any distinct $A,B \in \mathcal{F}$.   Equivalently, $f_L(n)$ is the maximum size of a subset of the hypercube $\left\{0, 1\right\}^n$ with pairwise Hamming distance in $L$. A classical result of Delsarte~\cite{Delsarte} shows that $f_L(n)\leq \sum_{i=0}^{|L|} \binom{n}{i}$. When $L=[s]$ for some positive integer $s$, a celebrated theorem of Kleitman~\cite{Kleitman} fully determines $f_{[s]}(n)$.
       \begin{theorem}[Kleitman's theorem]\label{kleit}
       	 \begin{equation*}
       	   f_{[s]}(n) \leq
       	    \left\{
       	   \begin{aligned}
       	   &\sum_{i=0}^{t} \binom{n}{i},
       	   &\text{ for}\ s=2t ;\\
       	   &2 \sum_{i=0}^{t} \binom{n-1}{i} ,  &\text{ for}\ s=2t+1.
       	   \end{aligned}
       	   \right.
       	 \end{equation*}
       \end{theorem}
       Both inequalities in Theorem~\ref{kleit} are tight. When $s = 2t$, the upper bound can be attained by a Hamming ball
       of radius $t$. For $s = 2t + 1$,  one optimal example is the Cartesian product of $\left\{ 0, 1\right\}$ and the $(n-1)$-dimensional Hamming ball of radius $t$.

       For Sperner families, Nagy and Patk\'{o}s~\cite{Nagy-Patkos} introduced the notion of $L$-close Sperner systems. A set system $\mathcal{F} \subseteq 2^{[n]}$ is said to be {\it $L$-close Sperner}, if for any pair of distinct sets $A,B$ in $\mathcal{F}$, the skew distance $\min\{|A \setminus B|,|B \setminus A|\} \in L$.
       Using  linear independence arguments, Nagy and Patk\'{o}s~\cite{Nagy-Patkos} obtained the following result.
       \begin{theorem}[\cite{Nagy-Patkos}]\label{thm:NP}
       	Let $L$ be a set of $s$ positive integers. If $\mathcal{F}\subseteq 2^{[n]}$ is $L$-close Sperner, then
       	\begin{equation*}
       	|\mathcal{F}| \leq \sum_{i=0}^{s} \binom{n}{i}.
       	\end{equation*}
       \end{theorem}
       Recently, Xu and Yip~\cite{Xu-Yip} improved on this result for $L=[s]$. They combined `push to the middle' technique and the linear independence argument to prove the following upper bound.
       \begin{theorem}[\cite{Xu-Yip}]\label{thm:Xu-Yip-[s]}
       	Let $L=[s]$ with $(n+1)/3 \leq s \leq n/2$. If $\mathcal{F} \subseteq 2^{[n]}$ is $L$-close Sperner, then
       	$$
       	|\mathcal{F}|\leq \sum_{i=3s-n}^s \binom{n}{i}.
       	$$
       \end{theorem}
       In particular, when $n$ is even and $s=n/2$, the above upper bound is the same as Sperner's theorem.

  \subsection{Cover-free families and a conjecture of Erd\H{o}s, Frankl and F\"{u}redi}\label{subsection-cover-free}
       Let $t$ be a positive integer. A family $\mathcal{F}\subseteq 2^{[n]}$ is called \emph{$t$-cover-free}
       if $A_0 \not\subseteq A_1 \cup A_2 \cup \dots \cup A_r$ holds for all distinct
       $A_0, A_1, \dots , A_r \in \mathcal{F}$ with $r \leq t$. Thus a Sperner family is a $1$-cover-free family.

       Let $T(n, t)$ denote the maximum cardinality of a $t$-cover-free family $\mathcal{F}$ over an $n$-element underlying set. This notion was introduced by Kautz and Singleton~\cite{Kautz} in 1964 concerning binary codes. They proved that
       \begin{equation*}\label{equa-kautz-singleton}
         \Omega(1/t^2)  \leq \frac{\log_2 T(n,t)}{n} \leq O(1/t).
       \end{equation*}
       This result was rediscovered several times in information theory, in combinatorics by Erd\H{o}s, Frankl, and F\"{u}redi~\cite{erdos-frankl-furedi}, and in group testing by Hwang and S\'{o}s~\cite{Hwang-Sos}. In 1982, D'yachkov and Rykov~\cite{D'yachkov} obtained
       \begin{equation}\label{equa-D'yachkov}
       \limsup_{n\rightarrow \infty} \frac{\log_2 T(n,t)}{n} \leq (2+o(1)) \frac{\log_2 t}{t^2}
       \end{equation}
       with a quite involved proof.
       In 1994, Ruszink\'{o}~\cite{Ruszinko} gave a purely combinatorial proof of $\frac{\log_2 T(n,t)}{n} \leq O(\frac{\log_2 t}{t^2})$. Shortly after that,
       F\"{u}redi~\cite{furedi1996} presented the following slightly weaker result but using  a very elegant proof.

       \begin{theorem}[\cite{furedi1996}]\label{thm-furedi}
       	 Let $t\geq 2$. If $\mathcal{F}\subseteq 2^{[n]}$ is a $t$-cover-free family, then we have
       	 \begin{equation*}
       	  |\mathcal{F}| \leq t+ \binom{n}{\lceil \frac{n-t}{\binom{t+1}{2}} \rceil}.
       	 \end{equation*}
       	 Thus we have $T(n,t) \leq t +2^{\frac{4\log_2 t +O(1)}{t^2} n} $.
       \end{theorem}

        {\color{black} All above results provide general upper bounds for $T(n,t)$. For lower bounds, $T(n,t)\geq n$ since the family of singletons
         $\{\{1\},\{2\},\dots,\{n\}\}$ is $t$-cover-free.  This motivates another important problem about $T(n,t)$, which is to determine the smallest $n$ such that $T(n,t)> n$ for any given $t$.  Denote this minimum $n$ by $N^*(t)$, that is, the minimum $n$ such that there exists a $t$-cover-free family $\mathcal{F} \subseteq 2^{[n]}$ with $|\mathcal{F}| > n$. The determination of $N^*(t)$ would help to determine $T(n,t)$ for certain parameters and is closely related to group testing.}

       In 1985, Erd\H{o}s, Frankl and F\"{u}redi~\cite{erdos-frankl-furedi} proposed the following conjecture.
       \begin{conjecture}[\cite{erdos-frankl-furedi}]\label{conj:erdos-frankl-furedi}
       	  $\lim_{t \rightarrow \infty} N^*(t)/t^2 =1$, or in an even stronger form $N^*(t)\geq (t+1)^2$.
       \end{conjecture}
       Note that when $q$ is a prime power, an affine plane of order $q$ induces a $(q-1)$-cover-free family $\mathcal{F}\subseteq 2^{[q^2]}$ of size $q^2+q$. Let $q$ be the smallest prime power no less than $t+1$. Then the existence of an affine plane of order $q$ gives rise to a $t$-cover-free family $\mathcal{F}\subseteq 2^{[q^2]}$ of size $q^2+q$. Hence $N^*(t) \leq (1+o(1))t^2$. The currently best known general lower bound is due to Shangguan and Ge~\cite{Shangguan2016-disjunct-matrices}, which is restated below.
       \begin{theorem}[\cite{Shangguan2016-disjunct-matrices}]\label{thm-shangguan-cover-free}
       	 {\color{black} For any $t\geq 1$, } $N^*(t) \geq \frac{15+\sqrt{33}}{24}t^2$.
       \end{theorem}
   Theorem~\ref{thm-shangguan-cover-free}  directly implies that if $n< \frac{15+\sqrt{33}}{24}t^2$ and $\mathcal{F}\subseteq 2^{[n]}$ is a $t$-cover-free family, then $|\mathcal{F}| \leq n$. This gives
   a tight upper bound since the family of all singleton sets is a $t$-cover-free family.
   We will prove that this family is the unique optimal family for $n< \frac{15+\sqrt{33}}{24}t^2$ in Section~\ref{section:pf-of-tight-bound}.

\section{Proof of Theorem~\ref{thm:2-wFP-main}}\label{section:pf-of-2wFP}
In this section we prove Theorem~\ref{thm:2-wFP-main}. Our proof is based on the tools in the last section and the relationship between Sperner families and wide-sense frameproof codes.

 Given a code $\mathcal{C}= \{\mathbf{c}^1,\mathbf{c}^2,\dots,\mathbf{c}^m\}\subseteq Q^n$ of size $m$. For any $1\le i\ne j\le m$, define $I(i,j)$ to be the {\it coincidence set} of   $\mathbf{c}^i$ and $\mathbf{c}^j$, that is, $$I(i,j)\triangleq\{k:c^{i}_k=c^{j}_k, k\in [n]\}.$$  The following result is simple but useful.

\begin{lemma}[\cite{Zhou-Zhou}, Lemma 2.7]{\label{lemma:symmetric-diff}}
	Let $\mathcal{C} = \{\mathbf{c}^1,\mathbf{c}^2,\dots,\mathbf{c}^m\}\subseteq Q^n$ be any code.  For any distinct $i,j,k\in [m]$, we have
	\begin{equation*}
	I\left( {i,j} \right) \cap I\left( {i,k} \right) \subseteq  I\left( {j,k} \right) \subseteq \overline{I(i,j) \Delta I(i,k)}.
	\end{equation*}
\end{lemma}

 For any $i\in [m]$, define $$\mathcal{X}_i\triangleq\{I(i,j): j\in [m]\setminus\{i\}\}\subseteq 2^{[n]}$$ to be the {\it coincidence family} generated by the codeword $\mathbf{c}^i\in\mathcal{C}$.
Stinson and Wei~\cite{Stinson-Wei1998} were the first to establish the relationship between Sperner families and wide-sense frameproof codes, and Panoui~\cite{panoui} further proved the following result. Here, a family $\mathcal{F}\subseteq 2^{[n]}$ is called
{\it non $2$-covering} if for every pair of sets $A,B \in\mathcal{F}$ we have $A\cup B \neq [n]$.

\begin{theorem}[\cite{panoui}, Lemma 6.3.2, Corollary 6.3.3]\label{thm-panoui} Let $\mathcal{C} = \{\mathbf{c}^1,\mathbf{c}^2,\dots,\mathbf{c}^m\}\subseteq Q^n$ be a code of size $m$. Then, $\mathcal{C}$ is a wide-sense $2$-frameproof code if and only if
	for each $i\in [m]$, the coincidence family $\mathcal{X}_i$ is a non $2$-covering Sperner family of size $m-1$.
\end{theorem}

To prove Theorem~\ref{thm:2-wFP-main}, we first show an upper bound of the size of a non $2$-covering Sperner family which contains a relatively small set.
\begin{proposition}\label{prop:sperner}
	Let $\epsilon\in (0,1/2)$ be a constant. Suppose that $\mathcal{F} \subseteq 2^{[n]}$ is a non $2$-covering Sperner family such that there exists some $A_0\in \mathcal{F}$ with $|A_0|\leq \epsilon n$. Then we have
	\begin{equation*}
	|\mathcal{F}| \leq \left(1-\frac{1}{\binom{n}{\lfloor \epsilon n\rfloor}}\right)\binom{n}{\lfloor\frac{n-1}{2}\rfloor}+1.
	\end{equation*}
\end{proposition}
\begin{proof} Let $\mathcal{F} \subseteq 2^{[n]}$ and $A_0\in \mathcal{F}$ be the  objects as stated.
	Since $\mathcal{F}$ is a Sperner family, by LYM inequality we have
	\begin{equation*}
	\frac{1}{\binom{n}{|A_0|}} + \sum_{A\in \mathcal{F}\setminus \{A_0\}} \frac{1}{\binom{n}{|A|}} =\sum_{A\in \mathcal{F}} \frac{1}{\binom{n}{|A|}} \leq 1.
	\end{equation*}
	Hence,
	\begin{equation*}
	\frac{|\mathcal{F}|-1}{\binom{n}{\lfloor \frac{n}{2} \rfloor}} \leq
	\sum_{A\in \mathcal{F}\setminus \{A_0\}} \frac{1}{\binom{n}{|A|}} \leq 1 - \frac{1}{\binom{n}{|A_0|}}
	\leq 1-\frac{1}{\binom{n}{\lfloor \epsilon n\rfloor}}.
	\end{equation*}
	This shows that
	\begin{equation*}
	|\mathcal{F}| \leq \left(1-\frac{1}{\binom{n}{\lfloor \epsilon n\rfloor}}\right)\binom{n}{\lfloor\frac{n}{2}\rfloor}+1.
	\end{equation*}
	When $n$ is odd, we are done since $\lfloor n/2\rfloor =\lfloor \frac{n-1}{2}\rfloor$. Now assume that $n$ is even.
	\begin{claim}\label{claim:sperner} For even $n$, there exists a Sperner family $\mathcal{F}' \subseteq \binom{[n]}{\leq \frac{n}{2}-1}$  such that $A_0\in \mathcal{F}'$ and $|\mathcal{F}'|\geq |\mathcal{F}|$.
	\end{claim}
	\begin{proof}
		Denote $\mx_{\mathcal{F}}\triangleq\max\{|A|: A\in \mathcal{F}\}$. For any $k\leq n$, denote $\mathcal{F}_{k}\triangleq\{ A\in \mathcal{F}: |A|=k \}$ and $\mathcal{F}_{\leq k}\triangleq\{ A\in \mathcal{F}: |A|\leq k \}$.

 We proceed by induction on $\mx_\mathcal{F}$. When $\mx_{\mathcal{F}} \leq \frac{n}{2}-1$, the claim is trivial.
		Assume now that $\mx_{\mathcal{F}} = \frac{n}{2}$. Note that $\mathcal{F}_{\frac{n}{2}} \subseteq \mathcal{F}$ is a non 2-covering family, hence $\mathcal{F}_{\frac{n}{2}}$ is an intersecting family. By Lemma~\ref{lemma-katona}, we have $|\partial \mathcal{F}_{\frac{n}{2}}| \ge |\mathcal{F}_{\frac{n}{2}}|$. Since $\mathcal{F}$ is Sperner, it is clear that $\partial \mathcal{F}_{\frac{n}{2}}$ and $\mathcal{F}_{\leq \frac{n}{2}-1}$ are disjoint. 
		Let $\mathcal{F}'=\mathcal{F}_{\leq \frac{n}{2}-1} \cup \partial \mathcal{F}_{\frac{n}{2}} $. Hence $|\mathcal{F'}| \geq |\mathcal{F}|$ and $A_0\in \mathcal{F}'\subseteq \binom{[n]}{\leq \frac{n}{2}-1}$. Further, it is easy to check that  $\mathcal{F'}$ is a non $2$-covering Sperner family, so the claim holds for $\mx_{\mathcal{F}} = \frac{n}{2}$.

		
		Now suppose that the claim holds for $\mx_\mathcal{F} < k$ with $k\geq \frac{n}{2}+1$.  Let $\mathcal{F} \subseteq 2^{[n]}$ be a non 2-covering Sperner family with $\mx_\mathcal{F} = k$ containing some $A_0\in \mathcal{F}$ with $|A_0|\leq \epsilon n$.
 	Since $k\geq \frac{n}{2}+1$, $\mathcal{F}_{k} \subseteq \mathcal{F}$ is intersecting, and thus 	$|\partial \mathcal{F}_{k}| \ge |\mathcal{F}_{k}|$  by Lemma~\ref{lemma-katona}. Let $\mathcal{F}_0=\mathcal{F}_{\leq k-1} \cup \partial \mathcal{F}_k$. By similar arguments, $\mathcal{F}_0$ is  a non $2$-covering Sperner family containing some $A_0\in \mathcal{F}_0$ with $|A_0|\leq \epsilon n$ and $|\mathcal{F}_0|\geq |\mathcal{F}|$. Since $\mx_{\mathcal{F}_0} \leq k-1$, we apply the induction to $\mathcal{F}_0$ to get a Sperner family $\mathcal{F}' \subseteq \binom{[n]}{\leq \frac{n}{2}-1}$ such that $A_0\in \mathcal{F}'$ and $|\mathcal{F}'|\geq |\mathcal{F}_0| \geq |\mathcal{F}|$. This completes the proof of the claim.
	\end{proof}
	By Claim~\ref{claim:sperner}, there exists a Sperner family $\mathcal{F}' \subseteq \binom{[n]}{\leq \frac{n}{2}-1}$ such that $A_0\in \mathcal{F}'$ and $|\mathcal{F}'|\geq |\mathcal{F}|$. Applying the LYM inequality to $\mathcal{F}'$, we obtain
	\begin{equation*}
	\frac{1}{\binom{n}{|A_0|}} + \sum_{A\in \mathcal{F}'\setminus \{A_0\}} \frac{1}{\binom{n}{|A|}} =\sum_{A\in \mathcal{F}'} \frac{1}{\binom{n}{|A|}} \leq 1.
	\end{equation*}
	Hence,
	\begin{equation*}
	\frac{|\mathcal{F}|-1}{\binom{n}{ \frac{n}{2} -1}} \leq
	\frac{|\mathcal{F}'|-1}{\binom{n}{ \frac{n}{2} -1}} \leq
	\sum_{A\in \mathcal{F}'\setminus \{A_0\}} \frac{1}{\binom{n}{|A|}} \leq 1 - \frac{1}{\binom{n}{|A_0|}}
	\leq 1-\frac{1}{\binom{n}{\lfloor \epsilon n\rfloor}}.
	\end{equation*}
	This shows that
	\begin{equation*}
	|\mathcal{F}| \leq \left(1-\frac{1}{\binom{n}{\lfloor \epsilon n\rfloor}}\right)\binom{n}{\frac{n}{2} -1}+1
	= \left(1-\frac{1}{\binom{n}{\lfloor \epsilon n\rfloor}}\right)\binom{n}{\lfloor\frac{n-1}{2}\rfloor}+1.
	\end{equation*}

	Thus for every $n$ we have
	\begin{equation*}
	|\mathcal{F}| \leq \left(1-\frac{1}{\binom{n}{\lfloor \epsilon n\rfloor}}\right)\binom{n}{\lfloor\frac{n-1}{2}\rfloor}+1.
	\end{equation*}
\end{proof}

Now we are ready to prove Theorem~\ref{thm:2-wFP-main}.
\begin{proof}[Proof of Theorem~\ref{thm:2-wFP-main}]
	Let  $\mathcal{C} = \{\mathbf{c}^1,\mathbf{c}^2,\dots,\mathbf{c}^m\}\subseteq Q^n$ be a wide-sense $2$-frameproof code and let $i\in[m]$.
	Consider the coincidence family $\mathcal{X}_i$, which is Sperner by Theorem~\ref{thm-panoui}. The proof can be divided into two cases.
	
	\textbf{Case 1.} Suppose that any distinct $A,B\in \mathcal{X}_i$ satisfy $|A\Delta B|<(1-\epsilon)n$. Since $\mathcal{X}_i$ is Sperner,  $|A \setminus B|>0$ and $|B \setminus A|>0$ for any distinct $A,B\in \mathcal{X}_i$. By Pigeonhole principle, we have
	\begin{equation*}
	0< \min\{|A \setminus B|,|B \setminus A|\} < \frac{(1-\epsilon)n}{2}
	\end{equation*}
	 for every distinct $A,B\in \mathcal{X}_i$. Hence $\mathcal{X}_i$ is $[\lceil \frac{1-\epsilon}{2}n\rceil -1]$-close Sperner. By Theorem~\ref{thm:NP} and Theorem~\ref{thm:Xu-Yip-[s]},
	\begin{equation*}
	m= |\mathcal{X}_i|+1
	\leq
	\sum_{3\lceil \frac{1-\epsilon}{2}n \rceil -n-3 \leq   i   \leq \lceil \frac{1-\epsilon}{2}n\rceil -1}
	\binom{n}{i} +1.
	\end{equation*}
	
	\textbf{Case 2.} There exist two distinct sets $A,B\in \mathcal{X}_i$ such that $|A\Delta B|\geq (1-\epsilon)n$. Write $A=I(i,j)$ and $B=I(i,k)$ for some $j\neq k$. By Lemma~\ref{lemma:symmetric-diff} we have
	\begin{equation*}
	|I(j,k)| \leq |\overline{I(i,j) \Delta I(i,k)}| \leq n- (1-\epsilon)n =\epsilon n.
	\end{equation*}
	So the family $\mathcal{X}_j$ is a Sperner family of size $m-1$ containing a set $I(j,k)$ with $|I(j,k)| \leq \epsilon n$. By Proposition~\ref{prop:sperner},
	\begin{equation*}
	m= |\mathcal{X}_j|+1 \leq \left(1-\frac{1}{\binom{n}{\lfloor \epsilon n\rfloor}}\right)\binom{n}{\lfloor\frac{n-1}{2}\rfloor}+2.
	\end{equation*}
	
	Combining the two cases, we obtain
	\begin{equation*}
	m  \leq
	\max \left\lbrace
	\sum_{3\lceil \frac{1-\epsilon}{2}n \rceil -n-3 \leq   i   \leq \lceil \frac{1-\epsilon}{2}n\rceil -1}
	\binom{n}{i} +1,
	\left(1-\frac{1}{\binom{n}{\lfloor \epsilon n\rfloor}}\right)\binom{n}{\lfloor\frac{n-1}{2}\rfloor}+2 \right\rbrace.
	\end{equation*}
\end{proof}
  We remark that Kleitman's theorem in Theorem~\ref{kleit} is also applicable here in Case 1, but would give a slightly weaker bound.

\section{Wide-sense frameproof codes and cover-free families}\label{section:cover-free-and-wFP}

    In this section we establish a connection between wide-sense frameproof codes and cover-free families, and then prove Theorem~\ref{thm:general-upper-bound} and Theorem~\ref{thm:general-upper-bound-2}.

Recall that in Theorem~\ref{thm-panoui}, a code  $\mathcal{C}= \{\mathbf{c}^1,\mathbf{c}^2,\dots,\mathbf{c}^m\} \subseteq Q^n$ is a wide-sense $2$-frameproof code if and only if
	for each $i\in [m]$, the coincidence family $\mathcal{X}_i$ is a non $2$-covering Sperner family of size $m-1$. This is equivalent to say that each family $\overline{\mathcal{X}_i}$ is an intersecting $1$-cover-free family.
 Now we extend Theorem~\ref{thm-panoui} to  establish a relation between wide-sense $t$-frameproof codes and cover-free families. 

    \begin{theorem}\label{thm:connection-cover-free-wFP} {\color{black}Let $m,t\geq 2$ be integers,} and
    let $\mathcal{C} = \{\mathbf{c}^1,\mathbf{c}^2,\dots,\mathbf{c}^m\}\subseteq Q^n$ be a code of size $m$. Then  $\mathcal{C}$ is a wide-sense $t$-frameproof code if and only if
    for each $i\in [m]$, the family $\overline{\mathcal{X}_i}$ is a $(t-1)$-cover-free family of size $m-1$.
    \end{theorem}
 
    In fact, if for each $i\in [m]$,  the family $\overline{\mathcal{X}_i}$ is a $(t-1)$-cover-free family of size $m-1$, then $\overline{\mathcal{X}_i}$ must be $t$-wise intersecting. See below. This means that the non $2$-covering property in  Theorem~\ref{thm-panoui} is not necessary.

\begin{lemma}\label{inter}
  Let $\mathcal{C} = \{\mathbf{c}^1,\mathbf{c}^2,\dots,\mathbf{c}^m\}\subseteq Q^n$ be any code.  If for each $i\in [m]$,  the family $\overline{\mathcal{X}_i}$ is a $(t-1)$-cover-free family of size $m-1$, then each $\overline{\mathcal{X}_i}$ must be $t$-wise intersecting.
\end{lemma}
\begin{proof}
	For any $i\in [m]$ and $\{i_1,i_2,\dots,i_s \} \subseteq [m]\setminus\{i\}$ with $s\leq t$, we need to show that $\overline{I(i,i_1)}\cap \dots \cap \overline{I(i,i_s)}\neq \emptyset$.
 Since the family $\overline{\mathcal{X}_{i_1}}=\{\overline{I(i_1,j)}:  j\in [m]\setminus\{i_1\}\}$ has $m-1$ distinct sets, we see that $\overline{I(i_1,i)},\overline{I(i_1,i_2)},\dots,\overline{I(i_1,i_s)}$ are $s$ distinct sets in the family $\overline{\mathcal{X}_{i_1}}$. Since $\overline{\mathcal{X}_{i_1}}$ is $(t-1)$-cover-free and $s-1\leq t-1$, we have
    	 \begin{equation*}
    	 \overline{I(i_1,i)} \not\subseteq \overline{I(i_1,i_2)} \cup \overline{I(i_1,i_3)}\cup \dots \cup \overline{I(i_1,i_s)}.
    	 \end{equation*}
    	 It follows that
    	 \begin{equation*}
    	   I(i_1,i_2) \cap I(i_1,i_3)\cap \dots \cap  I(i_1,i_s) \not\subseteq I(i_1,i).
    	 \end{equation*}
    	 Thus there exists some $l\in (I(i_1,i_2) \cap I(i_1,i_3)\cap \dots \cap  I(i_1,i_s))\setminus I(i_1,i)$ and hence $c_l^{i_1}=c_l^{i_2}=\dots=c_l^{i_s} \neq c_l^i$ by definition of coincidence sets. Thus $l\in \overline{I(i,i_1)}\cap \dots \cap \overline{I(i,i_s)}$ and hence $\overline{I(i,i_1)}\cap \dots \cap \overline{I(i,i_s)}\neq \emptyset$, completing the proof.
\end{proof}


To prove Theorem~\ref{thm:connection-cover-free-wFP}, we use the following equivalent definition of wide-sense $t$-frameproof codes.

   \begin{definition}\label{equiv-defn}
Let $m,t\geq 2$. A code $\mathcal{C} = \{\mathbf{c}^1,\mathbf{c}^2,\dots,\mathbf{c}^m\}\subseteq Q^n$ is a wide-sense $t$-frameproof code if and only if for any $i\in [m]$ and any
    	$\{i_1,i_2,\dots,i_s \} \subseteq [m]\setminus\{i\}$ with $s\leq t$,  there exists some $l\in [n]$ such that $c_l^i\neq c_l^{i_1}=c_l^{i_2}=\dots=c_l^{i_s}$.
    \end{definition}
    The equivalence between Definition~\ref{equiv-defn} and the original one in Definition~\ref{orgde} is clear. We remark that when $m\geq t+1$, for any  $\{i_1,i_2,\dots,i_s \} \subseteq [m]$ with $s\leq t$, the existence of $l\in [n]$ such that $ c_l^{i_1}=c_l^{i_2}=\dots=c_l^{i_s}$ is always true. Otherwise $m\leq t$ since any new codeword will belong to $\wdesc(\{c^{i_1},c^{i_2},\ldots,c^{i_s}\})$.

    \begin{proof}[Proof of Theorem~\ref{thm:connection-cover-free-wFP}]
    	 Suppose $\mathcal{C} = \{\mathbf{c}^1,\mathbf{c}^2,\dots,\mathbf{c}^m\}\subseteq Q^n$ is a wide-sense $t$-frameproof code. For any $i\in [m]$ and any distinct pair $j,k\in [m]\setminus\{i\}$, by Definition~\ref{equiv-defn}, there exists some $l\in [n]$ such that $c_l^j\neq c_l^k=c_l^i$; hence $l\in I(i,k)$ but $l\notin I(i,j)$, so $I(i,j)\neq I(i,k)$. It follows that for each $i\in [m]$ the family $\overline{\mathcal{X}_i}$ has exactly $m-1$ distinct members.

    	
    	Now we prove that each $\overline{\mathcal{X}_i}$ is $(t-1)$-cover-free. Assume that for some $i\in[m]$ there exists $\{j_0,j_1,\dots,j_s\}\subseteq [m]\setminus \{i\}$ with $s\leq t-1$ such that
    	 \begin{equation*}
    	  \overline{I(i,j_0)} \subseteq \overline{I(i,j_1)} \cup \overline{I(i,j_2)}\cup \dots \cup \overline{I(i,j_s)}.
    	 \end{equation*}
    	     	  This implies that $I(i,j_1)\cap I(i,j_2)\cap\dots\cap I(i,j_s) \subseteq I(i,j_0)$. However, by Definition~\ref{equiv-defn}, there exists some $l\in [n]$ such that $c_l^{j_0}\neq c_l^{i}=c_l^{j_1}=c_l^{j_2}=\dots=c_l^{j_s}$.
    	     	  Hence $l\in (I(i,j_1)\cap I(i,j_2)\cap\dots\cap I(i,j_s)) \setminus I(i,j_0)\neq \emptyset$, which is a contradiction. Therefore, for each $i\in [m]$, the family $\overline{\mathcal{X}_i}$ is $(t-1)$-cover-free.
    	
    The converse indeed has been obtained from the proof of Lemma~\ref{inter}, which has shown that  for any $i\in [m]$ and $\{i_1,i_2,\dots,i_s \} \subseteq [m]\setminus\{i\}$ with $s\leq t$,  there exists some $l\in [n]$ such that $c_l^i\neq c_l^{i_1}=c_l^{i_2}=\dots=c_l^{i_s}$. So $\mathcal{C}$ is a  wide-sense $t$-frameproof code by Definition~\ref{equiv-defn}.
    %
    \end{proof}


Now we are ready to prove Theorem~\ref{thm:general-upper-bound}.
         \begin{proof}[Proof of Theorem~\ref{thm:general-upper-bound}]
         	Let $t\geq 3$ and let $\mathcal{C} = \{\mathbf{c}^1,\mathbf{c}^2,\dots,\mathbf{c}^m\}\subseteq Q^n$ be a code of size $m$. By Theorem~\ref{thm:connection-cover-free-wFP}, for any $i\in [m]$, $\overline{\mathcal{X}_i}$ is a $(t-1)$-cover-free family of size $m-1$. It follows from Theorem~\ref{thm-furedi} that
         	\begin{equation*}
         	 m-1 \leq t-1+ \binom{n}{\lceil \frac{n-(t-1)}{\binom{(t-1)+1}{2}} \rceil}.
         	\end{equation*}
         	Thus we have $m \leq \binom{n}{\lceil \frac{n-t+1}{\binom{t}{2}} \rceil}+t$, as desired.         	
         \end{proof}
            Observe that here we can use the better bound in~\eqref{equa-D'yachkov} for cover-free families given by D'yachkov and Rykov~\cite{D'yachkov} instead of Theorem~\ref{thm-furedi}, and then Theorem~\ref{thm:general-upper-bound-2} follows.

\section{Proof of Theorem~\ref{thm:tight-bound}}\label{section:pf-of-tight-bound}
In this section we characterize optimal cover-free families to prove Theorem~\ref{thm:tight-bound}.

   For a family $\mathcal{F}=\{F_1,\dots,F_w\} \subseteq 2^{[n]}$ of size $w$, a binary matrix $M \in \{0,1\}^{n\times w}$ is called the incidence matrix of $\mathcal{F}$ if, for all $u\in [n]$ and $v\in [w]$,  $M(u,v)=1$ if and only if $u\in F_v$.
   We say an $n\times w$ binary matrix $M \in \{0,1\}^{n\times w}$ is \emph{$t$-disjunct} if it is the incidence matrix of some $t$-cover-free family.
    Or equivalently, an $n\times w$ binary matrix $M \in \{0,1\}^{n\times w}$ is called $t$-disjunct if for any $j\leq t$ the boolean sum of any $j$ columns does not contain any other column. 
   The notion of $t$-disjunct matrices was introduced by Kautz and Singleton~\cite{Kautz} in a different terminology when they were studying nonrandom
   	superimposed binary codes which may be used for information retrieval system, data communication and magnetic memories.

    For simplicity, we work with $t$-disjunct matrices rather than set systems. In the language of $t$-disjunct matrices, Theorem~\ref{thm-shangguan-cover-free} can be restated as follows.

    \begin{theorem}[\cite{Shangguan2016-disjunct-matrices}]\label{thm:shangguan-disjunct}
    	Suppose $M$ is an $n\times w$ $t$-disjunct matrix. If $n < \frac{15+\sqrt{33}}{24}t^2$, then $w\leq n$.
    \end{theorem}

    The bound is clearly tight since an $n\times n$ identity matrix is $t$-disjunct. However, Theorem~\ref{thm:shangguan-disjunct} does not characterize optimal matrices. In the following we show such a result, which will be used to prove Theorem~\ref{thm:tight-bound}.

    \begin{theorem}\label{thm:disjunct-n*n}
    	If $M$ is an $n\times n$ $t$-disjunct matrix with $1<n < \frac{15+\sqrt{33}}{24}t^2$, then $M$ is a permutation matrix. Or equivalently,  if $\mathcal{F}\subseteq 2^{[n]}$ is a $t$-cover-free family of size $n$ with $1<n < \frac{15+\sqrt{33}}{24}t^2$, then $\mathcal{F} = \{ \{1\},\{2\},\dots,\{n\} \}$.
    \end{theorem}

    The proof idea of Theorem~\ref{thm:disjunct-n*n} is motivated by the arguments in \cite{Shangguan2016-disjunct-matrices}. For a binary matrix $M$, the weight of a column $\mathbf{u}$, denoted by $|\mathbf{u}|$, is the number of $1$'s in $\mathbf{u}$. A column $\mathbf{u}$ of $M\in \{0,1\}^{n\times w}$ is called isolated if there exists some $r\in[n]$ such that $u_r=1$ but $u'_r=0$ for any other column $\mathbf{u}'$. For a
    given matrix $M\in \{0,1\}^{n\times w}$, a subset of $[n]$ is private if it belongs to a unique column. Here we abuse the notation to write  $T\subseteq \mathbf{u}$ to mean $u_i=1$ for all $i\in T$.  For each column $\mathbf{u}$, denote $P(\mathbf{u}):=\{T\in \binom{[n]}{2}: T\subseteq \mathbf{u} \text{ and } T \text{ is private}\}$ as the collection of private 2-subsets contained in $\mathbf{u}$ and denote $N(\mathbf{u}):=\{T\in \binom{[n]}{2}: T\subseteq \mathbf{u}  \text{ and } T \text{ is not private}\}$. Clearly $|P(\mathbf{u})|+|N(\mathbf{u})|=\binom{|\mathbf{u}|}{2}$.

    \begin{lemma}[\cite{Ruszinko}, Lemma 3.3]{\label{lemma:disjunct-delete}}
    	Let $M \in \{0,1\}^{n\times w}$ be a $t$-disjunct matrix and let $\mathbf{u}$ be a column of $M$ with weight $|\mathbf{u}|$. Then by deleting $\mathbf{u}$ and all rows intersecting it we get an $(n-|\mathbf{u}|)\times(w-1)$ $(t-1)$-disjunct matrix.
    \end{lemma}


     \begin{lemma}[\cite{Shangguan2016-disjunct-matrices}, Lemma 3.2]\label{lemma:shangguan-matching} Suppose $M\in \{0,1\}^{n\times w}$ is a $t$-disjunct matrix without isolated columns. Then for any column $\mathbf{u}$ with weight $|\mathbf{u}|=t+s$, where $1\leq s\leq t-1$, we have
     	\begin{equation*}
     	     	|N(\mathbf{u})| \leq\max\left\lbrace \binom{2s-1}{2},\binom{t+s}{2}-\binom{t+1}{2}\right\rbrace =
     	\begin{cases} \binom{t+s}{2}-\binom{t+1}{2}, & s\leq\frac{2t+2}{3},\\\binom{2s-1}{2}, & s\geq\frac{2t+2}{3}.
     	\end{cases}
     	\end{equation*}
     \end{lemma}


The following lemma was implicitly shown in \cite{Shangguan2016-disjunct-matrices}. In particular, the results for $n\times n$ and $n\times (n-1)$ $t$-disjunct matrices in the following lemma are crucial for us, which do not follow directly from the results in \cite{Shangguan2016-disjunct-matrices}. However, their method does work for these two cases and we follow their method to present a proof.
     \begin{lemma}{\label{lemma:key}}
     	 Let $t\geq 2$ and let $M \in \{0,1\}^{n\times w}$ be a $t$-disjunct matrix with $w\geq n-1$. If $n< \frac{15+\sqrt{33}}{24}t^2$, then there exists an isolated column in $M$.
     \end{lemma}

     \begin{proof}
     	Suppose for a contradiction that $M$ does not have any isolated column. First we show that each column of $M$ has weight at least $t+1$. Indeed, since every column in $M$ is not isolated, any $1$ in a column is contained in the same row of some other column. Hence if the weight of some column $\mathbf{u}$ is at most $t$, then $\mathbf{u}$ is contained in the boolean sum of at most $t$ other columns, which is impossible since $M$ is $t$-disjunct. This shows that the weight of every column in $M$ is at least $t+1$. 
     	We consider two cases below.
     	
     	\textbf{Case 1.} Suppose that every column $\mathbf{u}$ has weight $t+1\leq |\mathbf{u}| \leq \frac{15+\sqrt{33}}{12}t$. Write $|\mathbf{u}| = t+s$ where $1\leq s \leq \frac{3+\sqrt{33}}{12}t$. If $1\leq s\leq \frac{2t+2}{3}$, by Lemma~\ref{lemma:shangguan-matching},
     	  \begin{equation*}
     	   |P(\mathbf{u})| = \binom{|\mathbf{u}|}{2} - N(\mathbf{u}) \geq \binom{t+s}{2} - \left( \binom{t+s}{2}-\binom{t+1}{2} \right) = \binom{t+1}{2}
     	   \geq \frac{15+\sqrt{33}}{48}t^2.
     	  \end{equation*}
     	  If $\frac{2t+2}{3} < s \leq \frac{3+\sqrt{33}}{12}t$, by Lemma~\ref{lemma:shangguan-matching},
     	  \begin{equation*}
     	  |P(\mathbf{u})| = \binom{|\mathbf{u}|}{2} - N(\mathbf{u}) \geq \binom{t+s}{2} - \binom{2s-1}{2}
     	  \geq \frac{t^2+2ts-3s^2}{2}
     	  = \frac{1+2s/t - 3(s/t)^2}{2}t^2
     	  \geq \frac{15+\sqrt{33}}{48}t^2.
     	  \end{equation*}
     	  Thus $|P(\mathbf{u})| \geq \frac{15+\sqrt{33}}{48}t^2$ for every column $\mathbf{u}$. Observe that the union $\bigcup_\mathbf{u} P(\mathbf{u}) \subseteq \binom{[n]}{2}$ is  disjoint by definition of private subsets. Hence we have
     	  \begin{equation*}
     	  \binom{n}{2} \geq \sum_\mathbf{u} P(\mathbf{u}) \geq w\cdot \frac{15+\sqrt{33}}{48}t^2
     	  \geq (n-1) \cdot \frac{15+\sqrt{33}}{48}t^2.
     	  \end{equation*}
     	  It follows that $n\geq \frac{15+\sqrt{33}}{24}t^2$, contradicting to the assumption that $n< \frac{15+\sqrt{33}}{24}t^2$.
     	
     	  \textbf{Case 2.}  There exists some column $\mathbf{v}$ with weight at least $\frac{15+\sqrt{33}}{12}t$.
     	  Since $M$ is $t$-disjunct, it is clear that $|\mathbf{v}|< n$.
     	  By Lemma~\ref{lemma:disjunct-delete}, we can get a $(t-1)$-disjunct matrix $M'\in \{0,1\}^{(n-|\mathbf{v}|)\times (w-1)}$. Note that $|\mathbf{v}| \geq t+1 \geq 3$ since $t\geq 2$. Hence $n-|\mathbf{v}| \leq n-3 < w-1 $ by the assumption that $w\geq n-1$. However, since
     	  $$
     	  n-|\mathbf{v}| \leq n - \frac{15+\sqrt{33}}{12}t< \frac{15+\sqrt{33}}{24}t^2 - \frac{15+\sqrt{33}}{12}t \leq \frac{15+\sqrt{33}}{24}(t-1)^2,
     	  $$
     	  by Theorem~\ref{thm:shangguan-disjunct} we see that the number of columns of $M'$ satisfies
     	  $
     	  w-1 \leq n-|\mathbf{v}|
     	  $,
     	   which leads to a contradiction.
     	
     	   Therefore, there must exist an isolated column in $M$.
     \end{proof}

       \begin{proof}[Proof of Theorem~\ref{thm:disjunct-n*n}]
       	If $t=1$, there does not exist any integer $n$ such that $1<n < \frac{15+\sqrt{33}}{24}t^2$. Now let $t\geq 2$ be fixed.
       	We proceed by induction on $n$.
       	When $n=2$, it is trivial.
       	Suppose the result holds for every $k\times k$ $t$-disjunct matrix with $1<k<n < \frac{15+\sqrt{33}}{24}t^2$.
       	Let $M$ be an $n\times n$ $t$-disjunct matrix where $2<n < \frac{15+\sqrt{33}}{24}t^2$.
       	
       	By Lemma~\ref{lemma:key}, there exists an isolated column $\mathbf{u}$ in $M$. Without loss of generality, by rearranging the rows and columns, we can assume that $\mathbf{u}$ is the first column with the first row incident to $\mathbf{u}$ but not to any other column.
       	Then the matrix has the form
       	\begin{displaymath}       	
       	M =
       	\left( \begin{array}{c|ccc}
       	1   & 0 & \cdots  & 0  \\ \hline
       	a_2   &     &     &   \\
       	\vdots   &     & M_1  &   \\
       	a_{n}   &     &     &   \\
       	\end{array} \right).
       	\end{displaymath}
       	It is clear that $M_1$ is also a $t$-disjunct matrix by definition.  Note that $M_1$ is an $(n-1)\times (n-1)$ matrix. By the induction hypothesis, the matrix $M_1$ is a permutation matrix. We claim that $a_2=\dots=a_n=0$. Indeed, if $a_i=1$ for some $2\leq i\leq n$, let $\mathbf{u}'\neq \mathbf{u}$ be the column of $M$ having 1 in the $i$-th row. Then the column $\mathbf{u}'$ is covered by $\mathbf{u}$ since $\mathbf{u}'$ has exactly one $1$ entry, contradicting to the assumption that $M$ is $t$-disjunct. Thus $a_2=\dots=a_n=0$ and therefore $M$ is a permutation matrix, completing the proof.
       \end{proof}

       Now we prove Theorem~\ref{thm:tight-bound}.
       \begin{proof}[Proof of Theorem~\ref{thm:tight-bound}]
       	  Assume there exists a wide-sense $t$-frameproof code $\mathcal{C} = \{\mathbf{c}^1,\mathbf{c}^2,\dots,\mathbf{c}^{n+1}\}\subseteq Q^n$ with $t \geq 3$
       	  and $2 \leq n < \frac{15+\sqrt{33}}{24}(t-1)^2$.   Let $i\in[n+1]$.
       	  By Theorem~\ref{thm:connection-cover-free-wFP}, the family $\overline{\mathcal{X}_i} \subseteq 2^{[n]}$ is a  $(t-1)$-cover-free family of size $n$. Then $\overline{\mathcal{X}_i}$ gives an $n\times n$ $(t-1)$-disjunct matrix with $1<n < \frac{15+\sqrt{33}}{24}(t-1)^2$.
       	  By Theorem~\ref{thm:disjunct-n*n}, this matrix is a permutation matrix, and hence $overline{\mathcal{X}_i} = \{ \{1\},\{2\},\dots,\{n\} \}$. However, $\overline{\mathcal{X}_i}$ is $t$-wise intersecting by Lemma~\ref{inter}, which leads to a contradiction.
       \end{proof}

\section{Proof of Theorem~\ref{thm:tight-bound-iff}}\label{section:pf-of-tight-bound-iff}
In this section we prove Theorem~\ref{thm:tight-bound-iff}. We first use Lemma~\ref{lemma:key} to prove a stability result for $t$-cover-free families of size $n-1$, which extends Theorem~\ref{thm:disjunct-n*n} and may be of independent interest.

    \begin{theorem}\label{thm:n*(n-1)}
    	 Let $M$ be an $n\times (n-1)$ $t$-disjunct matrix with $2<n < \frac{15+\sqrt{33}}{24}t^2$. Then by permuting the rows of $M$ we have that the first $n-1$ rows form a permutation matrix of degree $n-1$ and the last row is arbitrary.
    \end{theorem}
    \begin{proof}
    	Let $t\geq 2$ be fixed.  We prove by induction on $n$. The base case $n=3$ is trivial.
    	
    	Now suppose that the result holds for every $k\times (k-1)$ $t$-disjunct matrix with $2<k<n < \frac{15+\sqrt{33}}{24}t^2$.
    	Let $M$ be an $n\times (n-1)$ $t$-disjunct matrix where $3<n < \frac{15+\sqrt{33}}{24}t^2$.
    	By Lemma~\ref{lemma:key}, there exists some isolated column $\mathbf{u}$ in $M$. Without loss of generality, we can permute the rows and columns, and assume that $\mathbf{u}$ is the first column with the first row incident to $\mathbf{u}$ but not to any other column. Then the matrix $M$ is of the form
    	\begin{displaymath}       	
    	M =
    	\left( \begin{array}{c|ccc}
    	1   & 0 & \cdots  & 0  \\ \hline
    	u_2   &     &     &   \\
    	\vdots   &     & M_1  &   \\
    	u_{n}   &     &     &   \\
    	\end{array} \right).
    	\end{displaymath}
    	Clearly $M_1$ is an $(n-1)\times (n-2)$ $t$-disjunct matrix. By the induction hypothesis we see that we can permute the rows such that the first $n-2$ rows of $M_1$ form a permutation matrix of degree $n-2$. Then $M$ has the form
    	\begin{displaymath}       	
    	M =
    	\left( \begin{array}{c|ccc}
    	1   & 0 & \cdots  & 0  \\ \hline
    	a_2   &     &     &   \\
    	\vdots   &     & M_2  &   \\
    	a_{n-1}   &     &     &   \\ \hline
    	a_{n}   &  b_2   & \cdots & b_{n-1}   \\
    	\end{array} \right),
    	\end{displaymath}
    	where $M_2$ is an $(n-2)\times (n-2)$ permutation matrix.

    	\textbf{Case 1.} Suppose $a_2=\dots=a_{n-1}=0$. Since $M_2$ is a permutation matrix of degree $n-2$, it is clear that the first $n-1$ rows of $M$ form a permutation matrix of degree $n-1$.
    	
    	\textbf{Case 2.} Suppose that there exists some $i_0 \in \{2,3,\dots,n-1\}$ such that $a_{i_0}=1$. Since $M_2$ is a permutation matrix, we can find a column $\mathbf{v}$ of $M_2$ such that $\mathbf{v}$ has entry $1$ in the $(i_0-1)$-th row and  $0$'s in other rows. Let $\mathbf{v}'$ be the column of $M$ corresponding to the column $\mathbf{v}$ of $M_2$, and assume $\mathbf{v}'$ is the $j_0$-th column of $M$.  Since $M$ is $t$-disjunct, the column $\mathbf{v}'$ is not contained in $\mathbf{u}$. Hence we have $a_n=0$ and $b_{j_0}=1$.
    	\begin{claim}
    		$b_j=0$ for all $j\in \{2,3,\dots,n-1\}\setminus\{j_0\}$.
    	\end{claim}
    \begin{proof}
    	   Indeed, if there exists some $j_1\in \{2,3,\dots,n-1\}\setminus\{j_0\}$ such that $b_{j_1}=1$, then $\mathbf{v}'$ is contained in the boolean sum of $\mathbf{u}$ and the $j_1$-th column of $M$, contradicting to the assumption that $M$ is $t$-disjunct. Thus $b_j=0$ holds for every $j\in \{2,3,\dots,n-1\}\setminus\{j_0\}$.
    	\end{proof}
       \begin{claim}
       	$a_i=0$ for all $i\in \{2,3,\dots,n-1\}\setminus\{i_0\}$.
       \end{claim}
    \begin{proof}
   	Assume there exists some $i_1\in \{2,3,\dots,n-1\}\setminus\{i_0\}$ such that $a_{i_1}=1$.
   	Since $M_2$ is a permutation matrix, we can find a column $\mathbf{w}\neq \mathbf{v}$ of $M_2$ such that $\mathbf{w}$ has entry $1$ in the $(i_1-1)$-th row and  $0$'s in other rows.
   	Let $\mathbf{w}'$ be the column of $M$ corresponding to the column $\mathbf{w}$ of $M_2$. Note that $a_n=0$ and $b_j=0$ for all $j\neq j_0$.
   	 Then $\mathbf{w}'$ is contained in  $\mathbf{u}$, contradicting to the assumption that $M$ is $t$-disjunct. Thus we have $a_i=0$ for every $i\in \{2,3,\dots,n-1\}\setminus\{i_0\}$.
   \end{proof}
By exchanging the $i_0$-th row and the last row of $M$, it is clear that the first $n-1$ rows become a permutation matrix of degree $n-1$, which completes the proof.
    \end{proof}

  Now we are in a position to prove Theorem~\ref{thm:tight-bound-iff}.
    \begin{proof}[Proof of Theorem~\ref{thm:tight-bound-iff}]
    	It is clear that given a code $\mathcal{C}\subseteq Q^n$, if its representation matrix  in standard form is a permutation matrix of degree $n$, then this code is wide-sense $t$-frameproof. Conversely, suppose now that $\mathcal{C} = \{\mathbf{c}^1,\mathbf{c}^2,\dots,\mathbf{c}^{n}\}\subseteq Q^n$ is a wide-sense $t$-frameproof code of size $n$, and suppose $M\in Q^{n\times n}$ is its representation matrix in standard form.
    	
    	 Let $i\in [n]$.
    	By Theorem~\ref{thm:connection-cover-free-wFP}, the family $\overline{\mathcal{X}_{i}} \subseteq 2^{[n]}$ is a  $(t-1)$-cover-free family of size $n-1$. Then the incidence matrix $N$ of $\overline{\mathcal{X}_{i}}$ is an $n\times (n-1)$ $(t-1)$-disjunct matrix with $3\leq n < \frac{15+\sqrt{33}}{24}(t-1)^2$.
    	By Theorem~\ref{thm:n*(n-1)}, we can permute the rows of $N$ such that the first $n-1$ rows form a permutation matrix of degree $n-1$. Without loss of generality, by permuting rows and columns, we can assume that $N$ is of the form
    \begin{displaymath}
    N =
    	\left( \begin{array}{ccc}	
    		    &  &      \\
    		   &  I_{n-1}   &        \\
    		   &     &      \\
    		\xi_1    &   \cdots  &  \xi_{n-1} \\
    	\end{array} \right),
    \end{displaymath}
    	where $I_{n-1}$ is the identity matrix of degree $n-1$. Since $\overline{\mathcal{X}_{i}}$ is $t$-wise intersecting by Lemma~\ref{inter}, it is easy to see that $\xi_1=\dots=\xi_{n-1}=1$.
    	Hence the incidence matrix of $\mathcal{X}_i$ is of the form
    	\begin{displaymath}
    		\left( \begin{array}{cccc}
    		0   & 1 & \cdots  & 1  \\
    		1   &  0   &   \cdots  &  1 \\
    		\vdots   & \vdots    &  \ddots & \vdots  \\
    		1   &   1  & \cdots    & 0  \\
    		0   &   0  & \cdots    & 0  \\
    	\end{array} \right)_{n\times (n-1)}.
    	\end{displaymath}
    	And therefore by permuting rows and columns the representation matrix $M$ of $\mathcal{C}$ has the form
    	\begin{displaymath}       	
    	M =
    	\left( \begin{array}{c|cccc}
    	a_1   & d_1 & a_1 & \cdots  & a_1  \\
    	a_2   &  a_2  & d_2  &   \cdots  & a_2  \\
    	\vdots   & \vdots & \vdots   & \ddots  &  \vdots \\
    	a_{n-1}   & a_{n-1}  & a_{n-1}  &   \cdots  &  d_{n-1} \\ \hline
    	a_{n}   &   b_1  & b_2 & \cdots    & b_{n-1}  \\
    	\end{array} \right),
    	\end{displaymath}
    	where the first column corresponds to the codeword $\mathbf{c}^i$, and we have
    	 $a_j\neq d_j$  and $a_n\neq b_j$ for each $j\in [n-1]$. Note that $M$ is in standard form and $n\geq 3$, we see that $a_1=\dots=a_{n-1}=0$ and $d_1=\dots=d_{n-1}=1$. Moreover, since $\mathcal{C}$ is wide-sense $t$-frameproof, it follows from Definition~\ref{equiv-defn} that for any
    	$\{j_1,j_2,\dots,j_s \} \subseteq [n-1]$ with $s\leq t$ we have $a_n\neq b_{j_1}=\dots=b_{j_s}$, so $a_n\neq b_1=\dots=b_{n-1}$. The assumption that $M$ is in standard form implies $a_n=1$ and $b_1=\dots=b_{n-1}=0$. Therefore $M$ is a permutation matrix, as desired.
     \end{proof}

\section{Concluding remarks}\label{section:conclusion}
 In this paper, we mainly investigate upper bounds for wide-sense frameproof codes.

 	 For wide-sense $2$-frameproof codes, using results from Sperner theory we prove a better upper bound, which significantly improves a result of Zhou and Zhou~\cite{Zhou-Zhou} when $n$ is large. Our result shows that the maximum size of a wide-sense $2$-frameproof code $\mathcal{C}\subseteq Q^n$ is much smaller than $\binom{n}{\lfloor \frac{n-1}{2}\rfloor}$ for large $n$. However, our bound is still $O(\binom{n}{\lfloor \frac{n-1}{2}\rfloor})$.  It would be of interest to further exponentially improve the upper bound. We strongly believe that the following is true.
 	\begin{conjecture}\label{conj:2-wFP}
 		Let $\mathcal{C}\subseteq Q^n$ be a wide-sense $2$-frameproof code of size $m$. For large $n$ we have
 		\begin{equation*}
 		m =
 		o\left( \binom{n}{\lfloor\frac{n-1}{2}\rfloor} \right).
 		\end{equation*}
 	\end{conjecture}

 	 For wide-sense $t$-frameproof codes, we establish a relationship between wide-sense $t$-frameproof codes and cover-free families. Using the bound for cover-free families we  provide a general upper bound. In particular, in the binary case our result improves the  best known  upper bound for binary $t$-frameproof codes given by Shangguan et al.~\cite{Shangguan2017}. However, the conditions for these cover-free families in the relationship are not fully explored. It would be interesting to further improve upon this result.
 	
 	 We also prove a tight upper bound for  wide-sense $t$-frameproof codes when the code length is at most $\frac{15+\sqrt{33}}{24}(t-1)^2$, which generalizes and improves a result of Ge et al.~\cite{Ge-Shangguan-Wang2019}. In Section~\ref{subsection-cover-free}, we define
 	$N^*(t)$ as the minimum $n$ such that there exists a $t$-cover-free family $\mathcal{F} \subseteq 2^{[n]}$ with $|\mathcal{F}| > n$.
 	Similarly, we can define $N_q(t)$ as the minimum $n\geq 2$ such that there exists a wide-sense $t$-frameproof code $\mathcal{C}\subseteq Q^n$ with $|\mathcal{C}| > n$.
 	Our result in Theorem~\ref{thm:tight-bound} shows that $N_q(t)\geq \lceil\frac{15+\sqrt{33}}{24}(t-1)^2\rceil\geq 4$ for $t\geq 3$. In \cite{Ge-Shangguan-Wang2019}, it was shown that for $t\geq 3$ it holds that $N^*(t-2)\leq N_2(t)\leq N^*(t)$. We can extend it to the following.
 	  \begin{proposition}\label{prop-relation-N(t)}
 	  	 For each $t\geq 3$, $N^*(t-2)+2\leq N_q(t) \leq N_2(t)\leq N^*(t)$.
 	  \end{proposition}
      \begin{proof}
      	 In \cite{Ge-Shangguan-Wang2019}, it was proved that $N^*(t-2)\leq N_2(t)\leq N^*(t)$. Since a binary wide-sense $t$-frameproof code is also a $q$-ary wide-sense $t$-frameproof code, we have $N_q(t) \leq N_2(t)$. It suffices to show $N^*(t-2)+2\leq N_q(t)$. Let $\mathcal{C} = \{\mathbf{c}^1,\mathbf{c}^2,\dots,\mathbf{c}^{n+1}\}\subseteq Q^n$ be a wide-sense $t$-frameproof code of size $n+1$ where $n=N_q(t)$. For each $i\in [n+1]$, by Theorem~\ref{thm:connection-cover-free-wFP}, $\overline{\mathcal{X}_i}\subseteq 2^{[n]}$ is a  $(t-1)$-cover-free family of size $n$. We claim that the incidence matrix $M$ of $\overline{\mathcal{X}_i}$ has a column of weight at least $2$. Indeed, otherwise there are two columns $\mathbf{u}, \mathbf{v}$ of $M$ with weight $|\mathbf{u}|=|\mathbf{v}|=1$ since $n=N_q(t)\geq 2$. Since the matrix $M$ is $(t-1)$-disjunct, the $1$'s in $\mathbf{u},\mathbf{v}$ are not in the same row. However, $\overline{\mathcal{X}_i}$ is $t$-wise intersecting by Lemma~\ref{inter}, which leads to a contradiction.
      	 Thus we can find a column $\mathbf{u}$ of weight at least $2$. By Lemma~\ref{lemma:disjunct-delete}, deleting $\mathbf{u}$ and all rows intersecting it yields an $(n-|\mathbf{u}|)\times(n-1)$  matrix which is $(t-2)$-disjunct. Since $n-|\mathbf{u}|<n-1$, we see that $N^*(t-2)\leq n-|\mathbf{u}|\leq N_q(t)-2$.
      \end{proof}
       In \cite{Ge-Shangguan-Wang2019}, the authors conjectured that $\lim_{t \rightarrow \infty} N_2(t)/t^2 =1$ which is equivalent to the weaker form of Conjecture~\ref{conj:erdos-frankl-furedi}. Our following conjecture is also equivalent to the weaker form of Conjecture~\ref{conj:erdos-frankl-furedi} by Proposition~\ref{prop-relation-N(t)}.
       \begin{conjecture}\label{conj-N_q(t)}
       	It holds that $\lim_{t \rightarrow \infty} N_q(t)/t^2 =1$.
       \end{conjecture}

 	 Let $t\geq 3$. We denote by $N_q'(t)$ be the smallest $n\geq 3$ such that there exists a wide-sense $t$-frameproof code in $Q^n$ of size $n$ whose representation matrix in standard form is NOT a permutation matrix. Theorem~\ref{thm:tight-bound-iff} shows that $N_q'(t) \geq \lceil\frac{15+\sqrt{33}}{24}(t-1)^2\rceil \geq 4$.  We can prove the following relationship between $N_q(t)$ and $N_q'(t)$.
 	\begin{proposition}\label{prop-relation-N_q'(t)}
 		 For each $t\geq 3$ we have $N_q'(t)\leq N_q(t)$.
 	\end{proposition}
 \begin{proof}
 	Suppose not, then for some $t\geq 3$, there exists a wide-sense $t$-frameproof code in $Q^n$ of size $n+1$ with $4\leq n= N_q(t)\leq N_q'(t)-1$. Let $M\in Q^{n\times (n+1)}$ be its representation matrix in standard form.
 		We claim that $M$ can be viewed as a binary matrix by symbol mapping. 
Otherwise, there exists some row of $M$  containing three different symbols, say $0,1,2$. Hence we can pick $n= N_q(t)>3$ columns of $M$ to form an $n\times n$ matrix $M'$ with some row containing $0,1,2$. Since $n= N_q(t)< N_q'(t)$, the standard form of $M'$ is a permutation matrix, leading to a contradiction. Thus $M$ is a binary matrix.
 		
 	Consider the submatrix $A$ formed by the first $n$ columns of $M$. Since $n<N_q'(t)$, the standard form of $A$ must be a permutation matrix. Note that $n= N_q(t)\geq 4$ and the binary matrix $M$ is in standard form. It is easy to see that $A$ is already in standard form. Hence the first $n$ columns of $M$ form a permutation matrix. Similarly, the last $n$ columns of $M$ also form a permutation matrix. It follows that the first and last columns of $M$ are identical, which is impossible.
 \end{proof}
   Combining Proposition~\ref{prop-relation-N_q'(t)} and Theorem~\ref{thm:tight-bound-iff} yields a new proof of Theorem~\ref{thm:tight-bound}.
    Our next conjecture (if true) implies Conjecture~\ref{conj-N_q(t)} {\color{black} since $N_q'(t)\leq N_q(t) \leq N^*(t) \leq (1+o(1))t^2$ for each $t\geq 3$.}
   \begin{conjecture}
   	It holds that $\lim_{t \rightarrow \infty} N_q'(t)/t^2 =1$.
   \end{conjecture}

 Lastly, it may be of interest to consider similar `wide-sense' analogues of other related structures, for example, that of separating hash families. For separating hash families, one may refer to \cite{Bazrafshan-Trung,Ge-Shangguan-Wang2019,Trung} and references therein. A code $\mathcal{C}= \{\mathbf{c}^1,\mathbf{c}^2,\dots,\mathbf{c}^m\}\subseteq Q^n$ is called a separating hash family of type $\left\lbrace w_1,w_2\right\rbrace$, denoted by SHF$(w_1,w_2)$, if for any disjoint subsets $C_1,C_2 \subseteq [m]$ with $|C_1|\leq w_1$ and $|C_2|\leq w_2$, there exists some $l\in [n]$ such that
 \begin{equation}\label{equa-shf}
  \left\lbrace c_l^j :j\in C_1\right\rbrace \cap \left\lbrace c_l^k :k\in C_2\right\rbrace = \emptyset.
 \end{equation}
 When $(w_1,w_2)=(1,t)$, an SHF$(1,t)$  is just a $t$-frameproof code.
 One may consider the following variant of this structure. We say $\mathcal{C}= \{\mathbf{c}^1,\mathbf{c}^2,\dots,\mathbf{c}^m\}\subseteq Q^n$ is a wide-sense SHF($w_1,w_2$), if for any disjoint subsets $C_1,C_2 \subseteq [m]$ with $|C_1|\leq w_1$ and $|C_2|\leq w_2$, there exists some $l\in [n]$ such that
 \begin{equation}\label{equa-wshf}
\left\lbrace c_l^j :j\in C_1\right\rbrace \cap \left\lbrace c_l^k :k\in C_2\right\rbrace = \emptyset \ \text{and} \ \left|\left\lbrace c_l^k :k\in C_2\right\rbrace\right|=1.
 \end{equation}
 Note that wide-sense $t$-frameproof codes can be viewed as the special case $(w_1,w_2)=(1,t)$ of the above notion.
 It is easy to see that for binary codes \eqref{equa-shf} and \eqref{equa-wshf} are equivalent.
  Similar to the arguments in Section~\ref{section:cover-free-and-wFP}, the above variant is actually closely related to a more general notion of cover-free families.
 A family $\mathcal{F}\subseteq 2^{[n]}$ is called $(r_1,r_2)$-cover-free, if for any $(s_1+s_2)$ different sets $A_1,\dots,A_{s_1}, B_1,\dots,B_{s_2}\in \mathcal{F}$ with $s_1\in [r_1],s_2\in [r_2]$, we have
 \[
 \bigcap_{i=1}^{s_1} A_i \not\subseteq \bigcup_{j=1}^{s_2} B_j.
 \]
 Clearly the classical definition of a $t$-cover-free family is the case $(r_1,r_2)=(1,t)$ of this definition. Moreover, observe that a family $\mathcal{F}\subseteq 2^{[n]}$ is $(r_1,r_2)$-cover-free if and only if  $\overline{\mathcal{F}}$ is an $(r_2,r_1)$-cover-free family.  For $(r_1,r_2)$-cover-free families and related objects, one may refer to the survey~\cite{Wei-survey}.

 The following result follows directly from the proof in Section~\ref{section:cover-free-and-wFP} and we omit its proof.
 Here a family $\mathcal{F}\subseteq 2^{[n]}$ is called non $t$-covering if $\cup_{A\in \mathcal{F}' }A \neq [n]$ for any $\mathcal{F}'\subseteq \mathcal{F}$ with $|\mathcal{F}'|\leq t$. So a family $\mathcal{F}\subseteq 2^{[n]}$ is $t$-wise intersecting if and only if $\overline{\mathcal{F}}$ is non $t$-covering.
 \begin{proposition}\label{wshf}
 	Let $w_1\geq 1, w_2\geq 2$ be integers and let $\mathcal{C} = \{\mathbf{c}^1,\mathbf{c}^2,\dots,\mathbf{c}^m\}\subseteq Q^n$ be a code of size $m\geq 2$. Then $\mathcal{C}$ is a wide-sense SHF($w_1,w_2$) if and only if
 	for each $i\in [m]$, the family $\overline{\mathcal{X}_i}$ is a $(w_1,w_2-1)$-cover-free family of size $m-1$,
 	or equivalently,  $\mathcal{X}_i$ is a $(w_2-1,w_1)$-cover-free family of size $m-1$.
 	Moreover, each $\overline{\mathcal{X}_i}$ is $w_2$-wise intersecting, or equivalently, each $\mathcal{X}_i$ is a non $w_2$-covering family.
 \end{proposition}

 By Proposition~\ref{wshf}, bounds on $(r_1,r_2)$-cover-free families surveyed in~\cite{Wei-survey} could provide bounds on the size of wide-sense SHFs.

\bigskip\noindent\textbf{Acknowledgements.}   We are grateful to Wenjie Zhong for helpful discussions.

\end{document}